\documentclass{amsart}
\usepackage{geometry}
\usepackage{verbatim}
\usepackage{url}
\usepackage{amsthm}
\usepackage{amssymb}
\usepackage{latexsym}
\usepackage{amsfonts}
\usepackage{amsmath}
\usepackage{amstext}
\usepackage{accents}
\usepackage{cancel}
\usepackage{eucal}
\usepackage{amscd}
\usepackage{hyperref}
\usepackage[shortalphabetic]{amsrefs}
\usepackage{graphicx}
\usepackage{subcaption}

\newtheorem{theorem}{Theorem}
\newtheorem{lemma}[theorem]{Lemma}
\newtheorem{corollary}[theorem]{Corollary}
\newtheorem{proposition}[theorem]{Proposition}
\newtheorem{defn}{Definition}
\theoremstyle{remark}
\newtheorem{remark}{Remark}
\newcommand{\RR}{\mathbb{R}}
\newcommand{\CC}{\mathbb{C}}

\newcommand{\ZZ}{\mathbb{Z}}
\newcommand{\NN}{\mathbb{N}}

\newcommand{\PP}{\mathbb{P}}

\newcommand{\E}{\mathrm{e}}

\newcommand{\vn}[1]{\lVert#1\rVert}

\begin{document}

\title{The curve-lengthening flow in inversive geometry} 
\author{Ben Andrews}
\address{Mathematical Sciences Institute, Australia National University.}
\email{Ben.Andrews@anu.edu.au}
\thanks{This research was partly supported by Discovery Projects grant DP15000097 and by Laureate Fellowship FL150100126 of the Australian Research Council.  The authors are grateful for the hospitality of the Yau Mathematical Sciences Center of Tsinghua University where some of the research was carried out.}
 \author{Glen Wheeler}
\address{School of Mathematics and Applied Statistics,
University of Wollongong}
 \email{glenw@uow.edu.au}

\begin{abstract}
We consider an invariant gradient flow for the invariant length functional for co-compact curves in inversive geometry, and prove that solutions exist for all time and converge to loxodromic curves, provided the initial curve is admissible (so that the invariant length element is well defined).
\end{abstract}

\maketitle

\section{Introduction}

In this paper we introduce and investigate a parabolic geometric evolution equation for curves which is invariant under inversive (M\"obius) transformations.  Precisely, this is the steepest ascent flow of the invariant length functional, with respect to an invariant $L^2$ inner product. This flow is well-posed for curves which are vertex-free, a condition which does not admit closed curves, so we consider it in the context of co-compact curves:  Curves which are periodic modulo the action of a fixed M\"obius transformation.
Our main result is that any admissible curve evolves under this flow for infinite time, ultimately converging smoothly to a loxodromic curve.  

\begin{theorem}
\label{TMmain}
Let $X_0:\ \RR\to\CC\PP^1$ be an admissible $L$-cocompact curve for which $\int Q^2\,ds<\infty$.  Then there exists a unique solution $X:\ \RR\times[0,\infty)\to\CC\PP^1$ of the inversive curve lengthening flow \eqref{eq:ICLF} which is smooth for $t>0$ and in $C^{0,1/12}\left([0,\infty),C^4(\RR,\CC\PP^1)\right)$ (in particular, the Gauss maps $G(u,t)$ are H\"older continuous on $\RR\times[0,\infty)$).

The solution $X:\RR\times[0,\infty)\rightarrow\CC\PP^1$
converges
exponentially fast in $C^p$ for every $p\geq 0$ to a limiting curve $X_\infty$, which is an $L$-cocompact loxodromic curve.
\end{theorem}

We refer the reader to Section \ref{sec:invDG} for the definition of $Q$ and notions of admissible and $L$-cocompactness.
Roughly speaking, (1) $Q$ is the first invariant of inversive geometry (and $ds$ here is the invariant measure), (2) admissible means that the Euclidean curvature is monotone increasing, and (3) $L$-cocompactness is a symmetry assumption: that the curve is generated by successive images of a period under a M\"obius transformation.

Our main result is an analogue in the setting of inversive geometry of some previous results, including the steepest ascent flow for the affine length in special affine geometry of planar curves \cite{ACLF}.  In that setting and present one, the evolution equations are both of high order and highly nonlinear, but their geometric invariance properties nevertheless make them accessible to a complete analysis.  In both cases, the nonlinearity is such that the pointwise admissibility condition (in the present case monotonicity of the curvature along the curve, and in \cite{ACLF} the strict convexity of the curve) is not violated along the evolution, despite the absence of maximum principles for such high order equations.

The paper is organised as follows.
We first introduce (in Section \ref{sec:invDG}) the required inversive differential geometry, including invariant tangent and normal vectors, arc length element and curvature. 
We use an invariant Gauss map, which is a map from the curve to the group $PSL(2,\CC)$ of M\"obius transformations.  From this we deduce analogues of the Serret-Frenet equations, and equations for the variation of inversive invariants under deformations of a curve.  In Section \ref{sec:evoleqn} we introduce the invariant flow and derive the associated evolution equations for various natural quantities, before proving long time existence and convergence through a series of integral estimates.
These estimates assume that the initial data is smooth.
In Section \ref{sec:rough} we are able to weaken the condition on the initial curve via a relaxation procedure, finally giving the full statement of Theorem \ref{TMmain}.

\section{Inversive Differential geometry of curves}\label{sec:invDG}

In this section we provide an introduction to the invariants of curves in inversive geometry, as well as a derivation of the evolution equations for these invariants under deformations of curves.  The essential details of this geometry were already worked out by Mullins \cite{Mullins} in 1917, at least in the context of planar curves.  More complete and geometric treatments were presented by Blaschke \cite{Blaschke}, Takasu \cite{Takasu} and Fialkow \cite{Fialkow}.
 Our treatment in this section emphasizes the `Gauss map' associated to the group of inversive transformations, and derives the invariants and their evolution equations from this.

\subsection{Inversive transformations}

The inversive or M\"obius transformations of the two-sphere $S^2$ can be conveniently described as follows:  We identify $S^2$ with the complex projective space ${\mathbb{CP}}^1$, which is the quotient $\left(\CC^2\setminus\{0\}\right)/\left(\CC\setminus\{0\}\right)$, or the space of complex lines in $\CC^2$.  The group $SL(2,{\CC})$ acts on ${\mathbb C}^2\setminus\{0\}$ by linear transformations, and this induces an action $\rho$ on $\CC\PP^1$.  The kernel of the action $\rho:\ SL(2,\CC)\to\text{diff}(\CC\PP^1)$ consists of $\{I,-I\}$ where $I$ is the identity transformation, so the M\"obius transformations are naturally identified with the group $PSL(2,\CC)=SL(2,\CC)/\{\pm I\}$.  This is a complex Lie group of dimension $3$, or a real Lie group of dimension $6$.  We will denote an element of $\CC^2$ as $z=\left(\begin{array}{c}z_1\\z_2\end{array}\right)$, and for $z\neq 0$ the corresponding element of $\mathbb{CP}^1$ will be denoted by $[z]$ or $\left[\begin{array}{c}z_1\\z_2\end{array}\right]$.

\subsection{The Inversive Gauss map}

In this section we derive the basic structure equations for (admissible) curves in inversive geometry.  We call a curve \emph{admissible} if the geodesic curvature is an increasing function along the curve --- it will follow from our considerations below that this condition is invariant under (orientation-preserving) inversive transformations.  The starting point is the identification of a canonical choice of inversive transformation which brings the curve near a given point to a normal form:

\begin{proposition}\label{prop:inv-norm-form}
Let $I$ be a connected one-dimensional manifold, and let $X:\ I\to \CC\PP^1$ be an admissible smooth immersion. Then 
for each $p\in I$ there exists a unique inversive transformation $G_p\in PSL(2,\CC)$ such that 
$$
G_p(X(p))=\left[\begin{array}{c}0 \\1\end{array}\right]
$$
and
$$
H(G_p(X(q))=o(|\varphi(q)-\varphi(p)|^4)
$$
as $q\to p$, where $H\left(\left[\begin{array}{c}x+iy \\1\end{array}\right]\right)= y-\frac16x^3$, and $\varphi$ is any local chart for $I$ near $p$.
\end{proposition}

\begin{proof}
Given $X:\ I\to\CC\PP^1$ there exists a horizontal lifting $\hat X:\ I\to S^3\subset\CC^2$ such that
$X(t) = \left[\hat X(t)\right]$, and $\hat X'(t)$ is orthogonal to the complex line $X(t)$ for each $t\in I$.  The parametrisation of $I$ can be chosen so that $\hat X'$ has length $1$ at each point.  It follows that in this parametrisation
\begin{equation}\label{eq:X2-HorizLift}
\hat X''= - \hat X +ik\hat X',
\end{equation}
where $k:\ I\to\RR$ is the geodesic curvature of the curve, and $i$ is scalar multiplication by the square root of $-1$.

The first step in constructing $G$ is to apply an $SU(2)$ transformation $G^1_x$ to bring $\hat X(x)$ to $\left(\begin{array}{c}0 \\1\end{array}\right)$, with tangent vector $\left(\begin{array}{c}1 \\0\end{array}\right)$:  Explicitly,
\begin{equation}\label{eq:G1x}
G^1_x = \left(\begin{array}{cc}(\hat X'(x))^* \\(\hat X(x))^* \end{array}\right),
\end{equation}
where $u^*$ is the Hermitian transpose of the column vector $u$.

Differentiating equation \eqref{eq:X2-HorizLift} gives expressions for higher derivatives of $\hat X$ as a linear combination of $\hat X$ and $\hat X'$.  This and the initial conditions $G^1_x\hat X(x)=\left(\begin{array}{c}0 \\1\end{array}\right)$ and $G^1_x\hat X' = \left(\begin{array}{c}1 \\0\end{array}\right)$ yield the following Taylor expansion for $G^1_x\hat X$:
\begin{equation}\label{eq:X4-Horiz-Lift}
G^1_x\hat X(x+t)= \left(\begin{array}{c}t+\frac{ik}{2}t^2+\frac{ik'-1-k^2}6t^3-\frac{3kk'+(2k+k^3-k'')i}{24}t^4+o(t^4) \\1-\frac{1}{2}t^2-\frac{ik}{6}t^3+\frac{1+k^2-2ik'}{24}t^4+o(t^4)\end{array}\right).
\end{equation}
Now consider the effect of applying a further $PSL(2,\CC)$ transformation $G^2_x=\left(\begin{array}{cc}a & 0 \\b+ic & a^{-1}\end{array}\right)$, where the vanishing of the upper right component is to ensure $G^2_xG^1_xX(x)=\left[\begin{array}{c}0 \\1\end{array}\right]$, and $a$ must be real to ensure the direction of the tangent vector in $\CC\PP^1$ is unchanged, and can be chosen to be positive by choosing the representative in $PSL(2,\CC)$.  The resulting Taylor expansion is
as follows:
\begin{align}\label{eq:Taylor-trans1}
G^2_xG^1_xX(x+t) &= \left[\begin{array}{c} at+\frac{aki}{2}t^2+\frac{a(ik'-1-k^2)}6t^3-\frac{a(3kk'+(2k+k^3-k'')i)}{24}t^4+o(t^4)\\
\frac1a+(b+ic)t+\frac{bki-a^{-1}-ck}{2}t^2-\frac{b(1+k^2)+ck'+i(a^{-1}k+ck^2-bk'+c)}{6}t^3+o(t^3)
\end{array}\right]\notag\\
&=\left[\begin{array}{c} 
p+iq\\
1
\end{array}\right]
\end{align}
where
$$p=a^2t-ba^3t^2+o(t^2)
$$
and
\begin{align*}q&=\frac{a^2(k-2ca)}{2}t^2
+a^2(\frac{k'}{6}-abk+2a^2bc)t^3\\
&\quad\null+\left(\frac{a^2}{24}(k''+8k-k^3)+\frac{a^3}{12}(7ck^2-8c-4bk')+\frac{3ka^4}{2}(b^2-c^2)+ca^5(c^2-3b^2)
\right)t^4+o(t^4).
\end{align*}
The function $H=q-\frac16p^3$ therefore has the following expansion:
\begin{equation}\label{eq:H-exp1}
H(x+t) = \left(\frac{k}{2}-ac\right)a^2t^2+o(t^2),
\end{equation}
so that $c=\frac{k}{2a}$.  Substituting this choice gives
$$
q=\frac16a^2k't^3+\frac{a^2(k''-8abk')}{24}t^4+o(t^4),
$$
so that
\begin{equation}\label{eq:H-exp3}
H(x+t) = \frac{a^2}{6}(a^4-k')t^3+o(t^3),
\end{equation}
and we must have $a=(k')^{1/4}$.  Finally this produces $p=\sqrt{k'}t-b(k')^{3/4}t^2+o(t^3)$ and $q=\frac{(k')^{3/2}}{6}t^3+\left(\frac{\sqrt{k'}k''}{24}-\frac{b(k')^{7/4}}{3}\right)t^4+o(t^4)$, and hence $H = \left(\frac1{24}\sqrt{k'}k''+\frac16b(k')^{7/4}\right)t^4+o(t^4)$, so that we must choose $b=-\frac{k''}{4(k')^{5/4}}$.
This gives 
\begin{equation}\label{eq:G2}
G^2_x = \left(\begin{array}{cc}(k')^{1/4} & 0 \\-\frac{k''}{4(k')^{5/4}}+\frac{k}{2(k')^{1/4}}i & (k')^{-1/4}\end{array}\right).
\end{equation}
This demonstrates that $G_x=G^2_xG^1_x$ is the unique inversive transformation bringing the curve to the required form.
\end{proof}

We call the map $G=G^2G^1:\ I\to PSL(2,{\mathbb C})$ the \emph{inversive Gauss map} of the curve.

\subsection{Invariant parameter, tangent and normal vectors}\label{subsec:InvParam}

If $X:\ I\to\CC\PP^1$ is an admissible curve, then for each $x\in I$ the invariant transformation $G_x$ is a diffeomorphism of $\CC\PP^1$ which maps $X(x)$ to $\left[\begin{array}{c}0 \\1\end{array}\right]$.  The tangent map $G_x\big|_*$ is a linear isomorphism from the tangent space at $X(x)$ to the tangent space at $\left[\begin{array}{c}0 \\1\end{array}\right]$, the inverse of which takes the vectors $\left(t\mapsto\left[\begin{array}{c}t \\1\end{array}\right]\right)'(0)$ and $\left(t\mapsto\left[\begin{array}{c}it \\1\end{array}\right]\right)'(0)$ respectively to vectors ${\mathcal T}$ and ${\mathcal N}$ which are tangential and orthogonal to the curve at $X(x)$.   These are given explicitly by the following formulae:
\begin{equation}\label{eq:inv-tang-nor}
{\mathcal T} = \left(k'\right)^{-1/2}{\mathbf t};\qquad {\mathcal N} = \left(k'\right)^{-1/2}{\mathbf n},
\end{equation}
where ${\mathbf n}$ and ${\mathbf t}$ are the unit normal and tangent vectors of the curve in $\CC\PP^1$ with the Fubini-Study metric, and the prime denotes the derivative with respect to arc length in the same metric. We refer to these as the inversive tangent and normal vectors to the curve $X$ at $x$.  By construction these are equivariant under inversive transformations:  If $\tilde X = L\circ X$ for some inversive transformation $L$, then the corresponding inversive tangent and normal vectors are given by $L_*({\mathcal T})$ and $L_*({\mathcal N})$. 

Dual to ${\mathcal T}$ is an invariant differential $ds = \left(k'\right)^{1/2}du$, where $du$ is the Fubini-Study arc length element.  Integrating this gives an invariant parameter $s$ along the curve, defined uniquely up to addition of a constant.

\begin{remark}\label{rmk:Euc-expr}
    The reader may find it convenient to work (locally) with curves projected from $\CC\PP^1$ to $\RR^2\simeq\CC$ by stereographic projection.  In this regard it is useful to note that the invariant parameter has a similar expression in terms of the geodesic curvature in $\RR^2$ of the projected curve:   $ds = \sqrt{\frac{\partial k}{\partial u}}du$, where $k$ is the geodesic curvature and $u$ the Euclidean arc length parameter along the projected curve.  This follows since the stereographic projection may be obtained as a limit of maps which combine M\"obius transformations with dilations, each of which preserves the invariant length element.  In particular, the stereographic image of an admissible curve in $\CC\PP^1$ has monotone geodesic curvature in $\RR^2$. 
\end{remark}

\subsection{Serret-Frenet equations}
\label{SerretFrenet}

In this section we derive the system of equations determining the rate of change of the Gauss map, and the corresponding `fundamental theorem of inversive plane curve theory', which implies that an admissible curve is determined (uniquely up to inversive transformations) by a single scalar function of the inversive length parameter, the `fundamental invariant' which we shall denote by $Q$.

A direct way to obtain this expression is to directly differentiate the expression for $G$ with respect to the inversive parameter.  This is straightforward but messy, and yields the following:
\begin{align}\label{eq:SF}
\frac{d}{ds}G &= \left(k'\right)^{-\frac12}\frac{d}{du}G\notag\\
&= \left(\begin{array}{cc}0 & -1 \\Q+\frac{i}{2} & 0\end{array}\right)G
\end{align}
where $Q=\frac14\frac{k^2}{k'}+\frac{5}{16}\frac{(k'')^2}{(k')^3}-\frac14\frac{k'''}{(k')^2}$.  

A more illuminating derivation of the inversive Serret-Frenet formula \eqref{eq:SF} is as follows: Fix $p$ in $I$.   Then by assumption we have
$$
G_p(X(p))=\left[\begin{array}{c}0 \\1\end{array}\right]
$$
and
$$
H\circ G_p(X(q)) = o(|\varphi(q)-\varphi(p)|^4)
$$
in any chart.  It follows that we can choose a chart $\varphi$ such that $\varphi(p)=0$ and 
$$
G_p(X(\varphi^{-1}(x))) = \left[\begin{array}{c}x+i\left(\frac16x^3-\frac{Q}{15}x^5+O(|x|^6)\right) \\1\end{array}\right]\quad\text{as}\ x\to 0
$$
for some $Q$.
Now we can write
$$
G_{\varphi^{-1}(x)}\circ G_p^{-1} = \left(\begin{array}{cc}1+xa & xb \\xc & 1-xa\end{array}\right)+O(x^2)\quad\text{as}\ x\to 0,
$$
for some $a,b,c\in\CC$.
Then we require that
\begin{align*}
\left[\begin{array}{c}0 \\1\end{array}\right] &= \left(G_{\varphi^{-1}(x)}\circ G_p^{-1}\right)\circ G_p\circ X(\varphi^{-1}(x))\\
&=\left(\begin{array}{cc}1+xa & xb \\xc & 1-xa\end{array}\right)
\left[\begin{array}{c}x \\1\end{array}\right]+O(x^2)\\
&=\left[\begin{array}{c}x(1+b) \\1-xa\end{array}\right]+O(x^2)\\
&=\left[\begin{array}{c}x(1+b) \\1\end{array}\right]+O(x^2),
\end{align*}
from which we conclude that $b=-1$.  Also, we have
\begin{align*}
G_{\varphi^{-1}(x)}\circ &X(\varphi^{-1}(x+y))=\left(G_{\varphi^{-1}(x)}\circ G_p^{-1}\right)\circ\left(G_p\circ X(\varphi^{-1}(x+y))\right)\\
&=\left(\begin{array}{cc}1+xa & -x \\xc & 1-xa\end{array}\right)
\left[\begin{array}{c}y+i\left(\frac16y^3-\frac{Q}{15}y^5\right)+x\left(1+i\left(\frac12y^2-\frac{Q}{3}y^4\right)\right) \\1\end{array}\right]+O(x^2+y^6)\\
&=\left[\begin{array}{c} y+i\left(\frac16y^3-\frac{Q}{15}y^5\right)+x\left(ay+\frac{i}2y^2+\frac{ia}{6}y^3-\frac{iQ}{3}y^4\right)
\\1+x\left(-a+cy+\frac{ic}{6}y^3
\right)\end{array}\right]+O(x^2+y^6)\\\\
&=\left[\begin{array}{c}y+i\left(\frac16y^3-\frac{Q}{15}y^5\right)+x\left(2ay+(\frac{i}{2}-c)y^2+\frac{ia}{6}y^3+\left(-\frac{iQ}{3}-\frac{ic}{6}\right)y^4
\right) \\1\end{array}\right]+O(x^2+y^6).
\end{align*}
From this we find that 
\begin{align*}
H\circ &G_{\varphi^{-1}(x)}\circ X(\varphi^{-1}(x+y))\\
&=\frac16y^3+xy\left(2a_2+(\frac12-c_2)y+\frac{a_1}{6}y^2+\left(-\frac{Q}{3}-\frac{c_1}{6}\right)y^3\right)+O(y^5)\\
&\quad\null-\frac{y^3}{6}\left(1+x\left(2a_1-c_1y+O(y^2)\right)\right)^3\\
&=x\left(2a_2y+\left(\frac12-c_2\right)y^2+\left(\frac{a_1}{6}-a_1\right)y^3+\left(-\frac{Q}{3}-\frac{c_1}{6}+\frac{c_1}{2}\right)y^4+O(y^5)\right)+O(x^2)
\end{align*}
where $a=a_1+ia_2$ and $c=c_1+ic_2$.  From this we conclude that $a_2=0$, $q_1=0$, $c_2=\frac12$ and $c_1=Q$, so that
$$
\frac{d}{dx}\left(G_{\varphi^{-1}(x)}\circ G_p^{-1}\right)|_{x=0} = \left(\begin{array}{cc}a & b \\c & -a\end{array}\right)=\left(\begin{array}{cc}0 & -1 \\Q+\frac{i}{2} & 0\end{array}\right).
$$
Noting that this derivative coincides with the derivative with respect to the invariant parameter at $p$, this
confirms \eqref{eq:SF} and identifies $Q$ with the coefficient of the lowest order term in the deviation of $G_p\circ X(I)$ from the model curve $\{y=\frac16x^3\}$.

\begin{remark}
    The Serret-Frenet equation \eqref{eq:SF} leads directly to a corresponding existence result:  For any smooth function $Q$ on $\RR$, and any $G_0\in PSL(2,\CC)$, there exists a unique admissible curve $X:\ \RR\to\CC\PP^1$ with fundamental invariant $Q(s)$ for each $s\in\RR$ for which the Gauss map $G$ satisfies $G(0)=G_0$.  Furthermore, if $Q$ is periodic on $\RR$ then $X$ is $L$-cocompact for some $L\in PSL(2,\CC)$.
\end{remark}
\begin{defn}
    An admissible curve $X$ is called a \emph{loxodromic curve} if the fundamental invariant $Q$ is constant along $X$.  
\end{defn}

Note that by Equation \eqref{eq:SF}, loxodromic curves have Gauss maps defined by one-parameter subgroups of $SL(2,\CC)$:  Since $Q$ is constant, a loxodromic curve has $G(s) = \exp\left(t\begin{bmatrix}0&-1\\Q+\frac{i}2&0\end{bmatrix}\right)G(0)$, and hence $X(s) = \exp\left(t\begin{bmatrix}0&-1\\Q+\frac{i}2&0\end{bmatrix}\right)X(0)$.  Particular examples include the curves $t\mapsto\begin{bmatrix} z_0r^t\E^{i\omega t}\\1\end{bmatrix}$, which project to the complex exponential spirals $t\mapsto z_0r^t\E^{i\omega t}$ in $\CC$ under the stereographic projection, for $0<r<1$ and $\omega>0$. Loxodromic curves are homogeneous in the sense that for any two points on a loxodromic curve there exists a M\"obius transformation preserving the curve which takes one point to the other. 
In Figure \ref{figLox1} we depict three loxodromic curves. and in Figure \ref{figLox2} we show their stereographic projections.

\begin{figure}[t]
    \centering
    \includegraphics[width=0.5\textwidth,trim=15cm 2cm 15cm 2cm,clip=true]{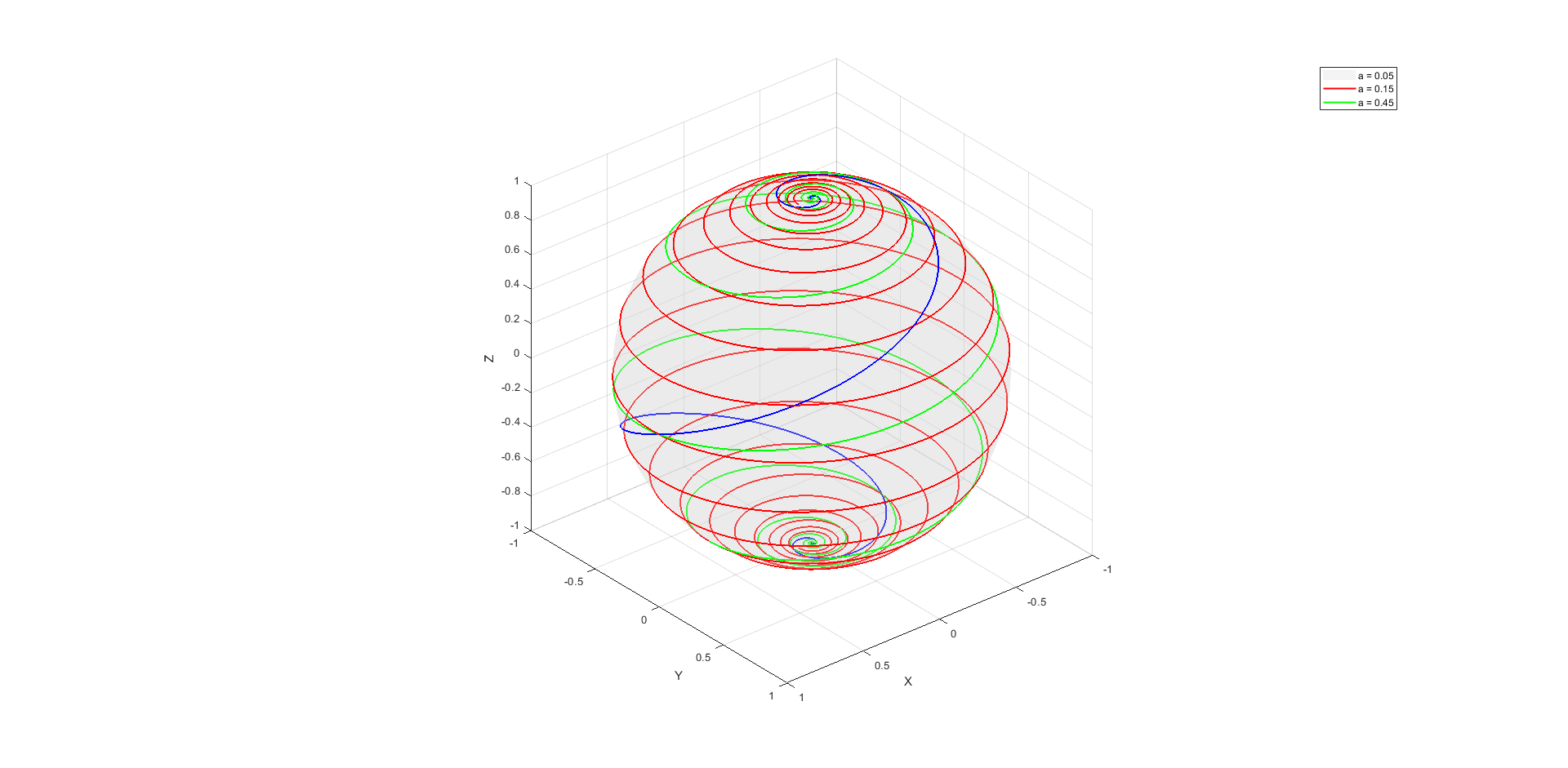}
    \caption{Three loxodromes on \(\mathbb{CP}^1\).
    They are determined by the maps $\left[\begin{array}{c}w_0(u) \\w_1(u)\end{array}\right] = \left[\begin{array}{c}\exp((a+i)u) \\1\end{array}\right]$ in homogeneous coordinates on $\mathbb{CP}^1$.}

  \label{figLox1}
\end{figure}

\subsection{Expansion of the Gauss map in invariant parameter}

We can now derive an expansion for the inversive Gauss map $G$, and hence also the embedding $X= G^{-1}\left(\left[\begin{array}{c}0 \\1\end{array}\right]\right)$, in terms of the inversive invariant $Q$ and the invariant arc length $s$:  For example, since $\partial_sG = MG$ with $M=\left(\begin{array}{cc}0 & -1 \\Q+\frac{i}{2} & 0\end{array}\right)$, we have
$$
\frac{\partial^2}{\partial s^2}G = (\partial_sM + M^2)G,
$$
and so on.  For our purposes it is more useful to obtain an expansion for the inverse transformation $G^{-1}$, which satisfies $\partial_sG^{-1} = -G^{-1}M$:  We find
\begin{align*}
G(0)G^{-1}(s) &= \left(\begin{array}{cc}1 & 0 \\0 & 1\end{array}\right)-s\left(\begin{array}{cc}0 & -1 \\Q+\frac{i}{2} & 0\end{array}\right)-\frac{s^2}2\left(\begin{array}{cc}Q+\frac{i}{2} & 0 \\Q_s & Q+\frac{i}{2}\end{array}\right)\\
&\quad\null-\frac{s^3}{6}\left(\begin{array}{cc}2Q_s & Q+\frac{i}{2} \\Q_{ss}-Q^2+\frac14-iQ & Q_s\end{array}\right)\\
&\quad\null -\frac{s^4}{24}\left(\begin{array}{cc}3Q_{ss}-Q^2+\frac14-iQ & 2Q_s \\Q_{sss}-4QQ_s-2iQ_s & Q_{ss}-Q^2+\frac14-iQ\end{array}\right)+O(s^5).
\end{align*}
\begin{figure}[t]
  \centering
  
  \begin{subfigure}[b]{0.3\textwidth}
    \centering
    \includegraphics[width=\textwidth,trim=13cm 0cm 13cm 0cm,clip=true]{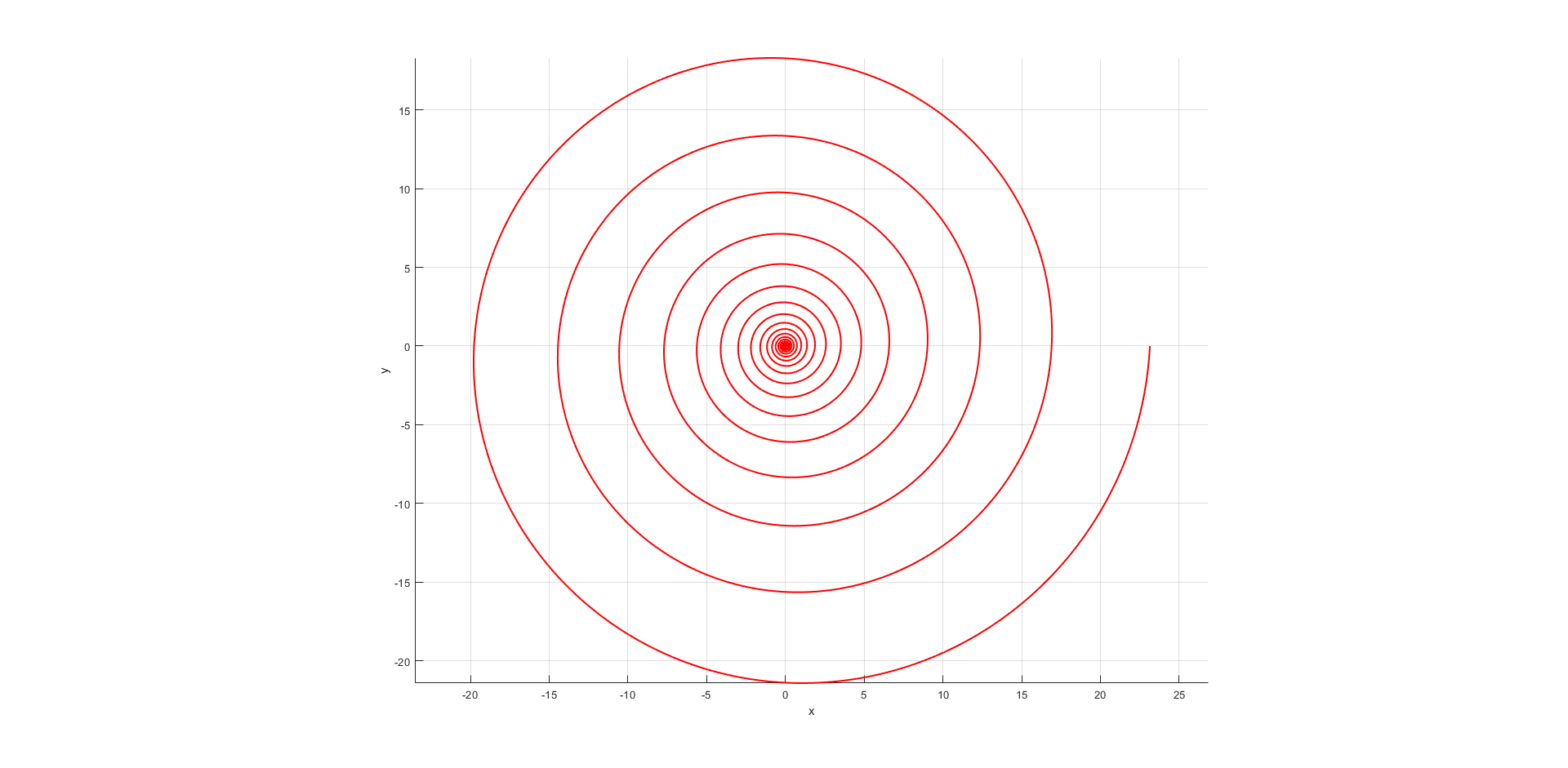}
    \caption{\(a = 0.05\)}
  \end{subfigure}
  \quad 
  \begin{subfigure}[b]{0.3\textwidth}
    \centering
    \includegraphics[width=\textwidth,trim=13cm 0cm 13cm 0cm,clip=true]{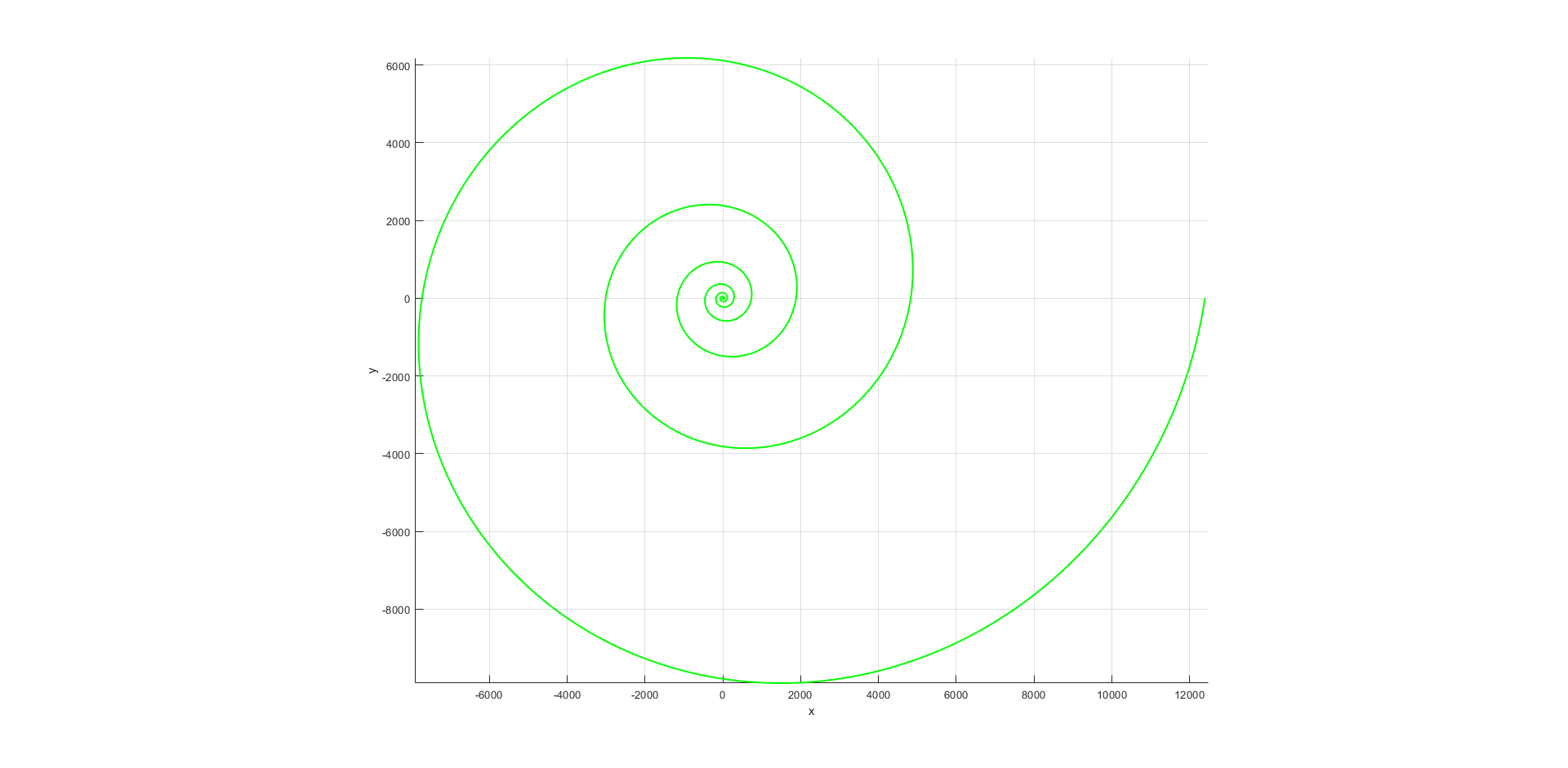}
    \caption{\(a = 0.15\)}
  \end{subfigure}
  \quad 
  \begin{subfigure}[b]{0.3\textwidth}
    \centering
    \includegraphics[width=\textwidth,trim=13cm 0cm 13cm 0cm,clip=true]{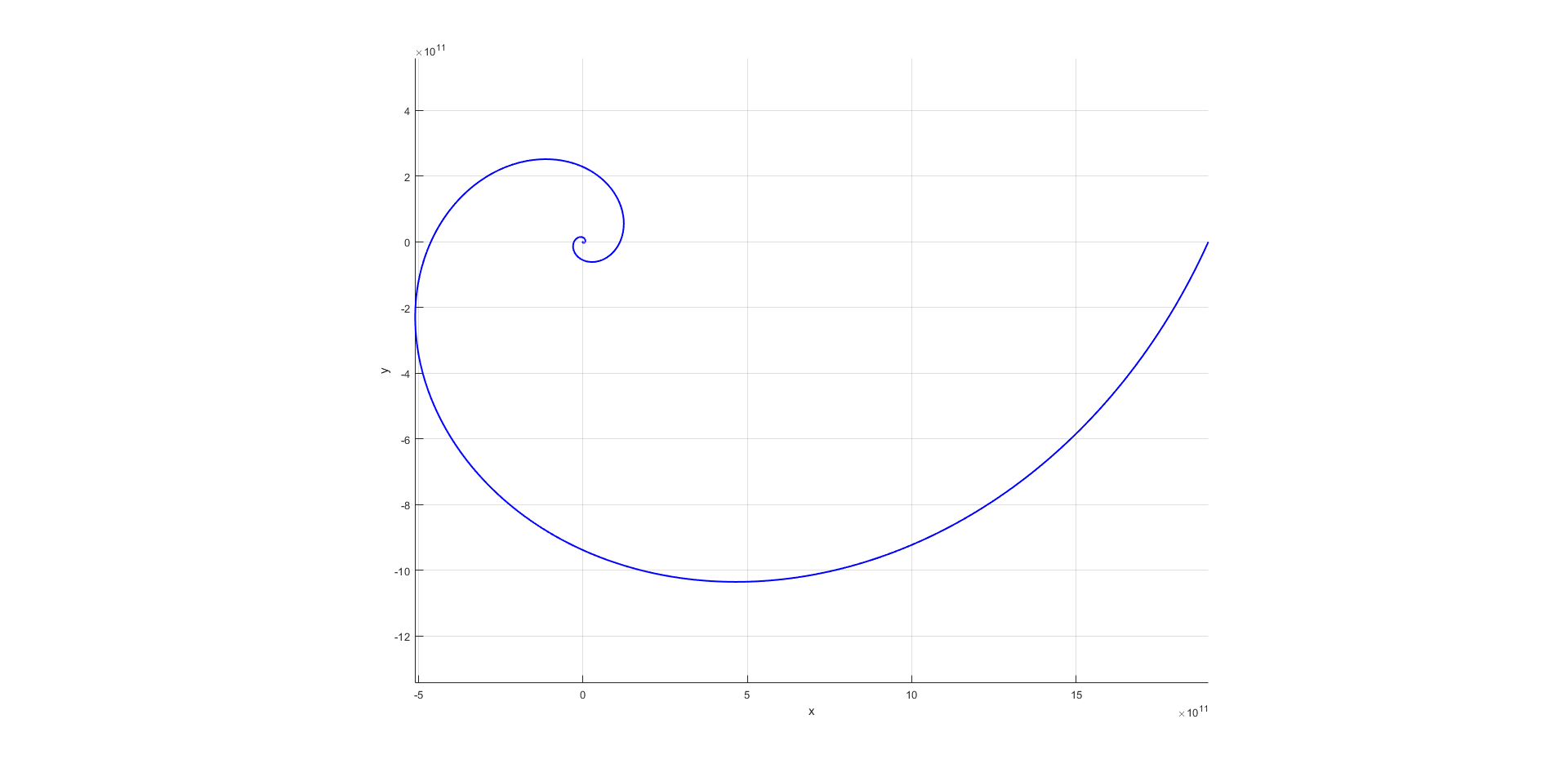}
    \caption{\(a = 0.45\)}
  \end{subfigure}

  \caption{Stereographic projection of the curves from $\mathbb{CP}^1$ in Figure \ref{figLox1}.
In the complex plane their parametrisation is $u\mapsto \exp((a + i)u)$.}
  \label{figLox2}
\end{figure}

\subsection{Variation equations}

We will consider `normal' variations of the form 
\begin{equation}\label{eq:normalvar}
\frac{\partial}{\partial t}X(p,t) = f(p){\mathcal N}(p,t),
\end{equation}
where $f$ is a smooth function on $I$ and ${\mathcal N}$ is the inversive normal vector.  We can compute the derivative with respect to $t$ at a fixed time $t=0$ of the inversive Gauss map $G$ by using the expansion in the inversive parameter derived above, as follows:  Fix $p\in I$, and let $s$ be the inversive parameter for the embedding $X(.,0)$ with $s(p)=0$.  Then ${\mathcal N}(s)$ is the image under $G(s,0)^{-1}$ of the vector $\left(t\mapsto \left[\begin{array}{c}ti \\1\end{array}\right]\right)'(0)$, so we have
$$
X(s,t) = G(s,0)^{-1}\left[\begin{array}{c}tf(s)i \\1\end{array}\right]+O(t^2).
$$
It follows that
\begin{align*}
&G(0,0)X(s,t) = (G(0,0)G(s,0)^{-1}\left[\begin{array}{c}tf(s)i \\1\end{array}\right]+O(t^2)\\
&=\left[\begin{array}{c}s+\frac{i+2Q}{6}s^3-\frac{Q_s}{3}s^4 \\1\end{array}\right]+t\left[\begin{array}{c}-\frac{fs^2}{2}-\frac{f_ss^3}{2}+\frac{(f^{(4)}+12Qf_{ss}+8Q_sf_s-6f_{ss}-126Qf)s^4}{24}\\0\end{array}\right]\\
&\null
+it\left[\begin{array}{c}f+f_ss+\frac{f_{ss}+2Qf}{2}s^2+\frac{f^{(3)}+6Qf_s+2Q_sf}{6}s^3+\frac{2Q_{ss}+16Q^2-4)f}{24}s^4 \\0\end{array}\right]+O(t^2+|s|^5)
\end{align*}
Now we can write $G(0,t)G(0,0)^{-1}=\left(\begin{array}{cc}1 & 0 \\0 & 1\end{array}\right)+t\left(\begin{array}{cc}A & B \\C & -A\end{array}\right)+O(t^2)$, so that
\begin{align*}
\left[\begin{array}{c}0 \\1\end{array}\right] &= G(0,t)X(0,t) = (G(0,t)G(0,0)^{-1})\circ G(0,0)X(0,t)
\\ &= \left(\begin{array}{cc}1+tA & tB \\tC & 1-tA\end{array}\right)\left[\begin{array}{c}tfi \\1\end{array}\right]+O(t^2)
\end{align*}
from which we conclude that $B=-fi$.  Further, we must have $0 = H\circ G(0,t)X(s,t)+O(s^5)$ as $s\to 0$.  
From the expression above, writing $G(0,0)X(s,t) = \left[\begin{array}{c}x+iy \\1\end{array}\right]$, we find
\begin{align*}
x &=s+t\left(2A_1s-\frac{f+2C_1}{2}s^2\right)+O(s^3);\\
y&=\frac16s^3+t\left((f_s+2A_2)s+\frac{f_{ss}+2Qf-2C_2}{2}s^2+\frac{f^{(3)}+6Qf_s+2Q_sf+2A_1+4QA_2}{6}s^3\right.\\
&\null\phantom{\frac16s^3+\qquad}\left.+\frac{f^{(4)}+12Qf_{ss}+8Q_sf_s+(16Q^2+2Q_{ss}-4)f-16C_2Q+4A_2Q_s-8C_1}{24}s^4
\right)+O(s^5)
\end{align*}
where $A=A_1+iA_2$ and $C=C_1+iC_2$.  This leads to the equation
\begin{align*}
0 &= y-\frac16x^3+O(s^5)\\
&=t\left((f_s+2A_2)s+\frac{f_{ss}+2Qf-2C_2}{2}s^2+\frac{f^{(3)}+6Qf_s+2Q_sf-4A_1+4QA_2}{6}s^3\right.\\
&\quad\null\qquad\left.+\frac{f^{(4)}+12Qf_{ss}+8Q_sf_s+(16Q^2+2Q_{ss}+2)f-16C_2Q+4A_2Q_s+4C_1}{24}s^4\right)+O(s^5)
\end{align*}
We solve for $A_1$, $A_2$, $C_1$ and $C_2$ to find $A_2=-\frac12f_s$, $C_2=\frac12f_{ss}+Qf$, $A_1 = \frac14f^{(3)}+Qf_s+\frac12Q_sf$, and
$C_1=-\frac14f^{(4)}-Qf_{ss}-\frac32Q_sf_s-\frac12(Q_{ss}+1)f$, implying that under the normal variation \eqref{eq:normalvar} the inversive Gauss map $G$ evolves according to
\begin{equation}\label{eq:dGdt}
\frac{\partial}{\partial t}G = TG,
\end{equation}
where
$$
T = \left(\begin{array}{cc}\frac14f^{(3)}+Qf_s+\frac12Q_sf-\frac{i}{2}f_s & -fi \\-\frac14f^{(4)}-Qf_{ss}-\frac32Q_sf_s-\frac12(Q_{ss}+1)f+\left(\frac12f_{ss}+Qf\right)i & -\frac14f^{(3)}-Qf_s-\frac12Q_sf+\frac{i}{2}f_s\end{array}\right).
$$

Finally, we can deduce variation equations for the inversive arc length and for the invariant $Q$, as follows:  Differentiating the equation $\partial_sG = MG$ with respect to $t$, we obtain
$$
\partial_t\partial_sG = (\partial_tM+MT)G.
$$
Differentiating $\partial_tG=TG$ with respect to $s$ gives
$$
\partial_s\partial_tG = (\partial_sT+TM)G.
$$
Writing $[\partial_t,\partial_s]={\mathcal C}\partial_s$, we derive the equation
$$
\partial_tM-{\mathcal C}M = \partial_sT+[T,M].
$$
The left-hand side is given by
$$
\left(\begin{array}{cc}0 & {\mathcal C} \\\partial_tQ-{\mathcal C}(Q+\frac{i}{2}) & 0\end{array}\right),
$$
so we can read off from this the quantity ${\mathcal C}$ and the time derivative of $Q$:
\begin{align*}
{\mathcal C} &= -\frac12 f^{(3)}-2Qf_s-Q_sf;\\
\partial_tQ &= -\frac14f^{(5)}-2Qf^{(3)}-\frac52Q_sf_{ss}-(4Q^2+2Q_{ss}+1)f_s-(2QQ_s+\frac12Q^{(3)})f.
\end{align*}
The quantity ${\mathcal C}$ determines the evolution of the arc length element $ds$, since if we write $ds=\rho du$ (for any fixed parametrisation $u$ of $I$) then we have $\partial_s=\rho^{-1}\partial_u$ and $[\partial_t,\partial_u]=0$, so that 
$$
{\mathcal C}\partial_s = [\partial_t,\partial_s] = \partial_t(\rho^{-1})\partial_u = -\frac{\partial_t\rho}{\rho}\partial_s,
$$
implying that $\partial_tds = \partial_t(\rho du) = -{\mathcal C}\rho du=-{\mathcal C}ds$. 

\subsection{Co-compactness}
\begin{figure}[t]
    \centering
    \includegraphics[width=0.4\textwidth,trim=1cm 1cm 1cm 1cm,clip=true]{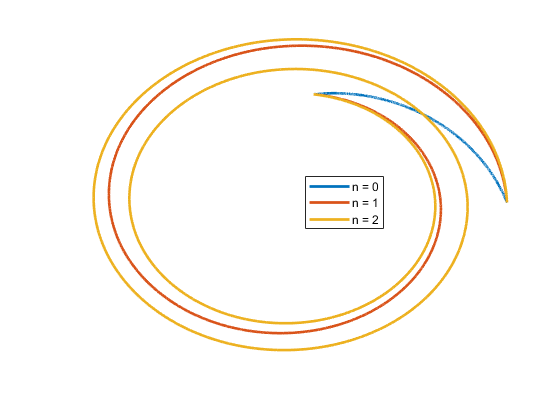}
    \caption{An illustration of three loxodromic curves with $n=0, 1$ and $2$. The integer $n$ counts the winding in each homotopy class.}

  \label{figLoxN}
\end{figure}

For $L\in PSL(2,\CC)$, we say that an admissible curve $X:\RR\to{\CC\PP}^1$ is $L$-co-compact if  $L(X(\RR))=X(\RR)$, with $L(X(s))\neq X(s)$ for every $s\in\RR$.  
It follows (after replacing $L$ by $L^{-1}$ if necessary) that there exists $\ell>0$ such that $L(X(s))=X(s+\ell)$ for all $s\in\RR$, where $s$ is the invariant parameter defined in section \ref{subsec:InvParam}.  Furthermore, the invariant $Q$ and its derivatives with respect to $s$ are invariant under $L$, and the inversive tangent ${\mathcal T}$ and normal ${\mathcal N}$ are equivariant, in the sense that
$L_*|_{X(s)}({\mathcal T}(s)) = {\mathcal T}(s+\ell)$ and $L_*|_{X(s)}({\mathcal N}(s))={\mathcal N}(s+\ell)$.

We observe that (up to a M\"obius transformation) any co-compact admissible curve is periodic with respect to a combination of shrinking and rotation:

\begin{lemma}\label{lem:symmetry}
If $X$ is an admissible $L$-co-compact curve, then there exists a M\"obius transformation $T$ so that $L = T^{-1}\circ S\circ T$, where $S = \left(\begin{matrix} \lambda &0\\0&\lambda^{-1}\end{matrix}\right)$ with $0<|\lambda|<1$.
\end{lemma}

\begin{proof}
    If $L$ has distinct eigenvalues, then $L$ can be diagonalised by a $GL(2,\CC)$ transformation, hence by an $SL(2,\CC)$ transformation, and exchanging the eigenvalues by conjugating by $\begin{pmatrix}0&i\\i&0\end{pmatrix}$ if necessary gives the required form provided the eigenvalues of $L$ do not lie on the unit circle.

    The remaining possibilities are where $L=T^{-1}\circ S\circ T$ where $S$ is either a unitary matrix $\begin{pmatrix} \mathrm{e}^{i\theta}&0\\0&\mathrm{e}^{-i\theta}\end{pmatrix}$
    or $\begin{pmatrix} 1&c\\0&1\end{pmatrix}$ for $c\in\CC$.  In the first case, $S$ acts on $TX$ by isometries, implying that $k(s+\ell)=k(s)$, contradicting the admissibility of $TX$ (hence of $X$) since the curvature must be monotone.  In the second case, $S$ acts by translations on the stereographic projection of $X$ into $\CC$, again contradicting the admissibility of $X$ (since the geodesic curvature in $\RR^2$ must be monotone as noted in Remark \ref{rmk:Euc-expr}).
\end{proof}

\begin{corollary}
    The homotopy classes of $L$-cocompact admissible curves are determined by $L$ and an integer $n\in\mathbb{N}$, defined by $\int_0^{\ell} k\,du = 2\pi n+2\theta$ along $PTX$, where $P$ is the steregraphic projection $\begin{bmatrix} z\\1\end{bmatrix}\mapsto z$, $T$ is the M\"obius transformation provided by Lemma \ref{lem:symmetry}, $du$ is the Euclidean arc length element and $k$ the Euclidean curvature along $PTX$, and $\lambda = r\E^{\theta}$ is the eigenvalue of $L$ inside the unit circle, with $\theta\in[0,\pi)$.
\end{corollary}
\begin{proof}
 Note that the restriction $\theta\in[0,\pi)$ can be fulfilled since in $PSL(2,\CC)$ we can multiply $L$ by $-1$ without changing the corresponding M\"obius transformation.
    By Lemma \ref{lem:symmetry}, $TX(s+\ell) = STX(s) = \begin{pmatrix}
        \lambda&0\\0&\lambda^{-1}
    \end{pmatrix}TX(s)$, so that $PTX(s+\ell) = \lambda^2PTX(s)=r^2\E^{2i\theta}PTX(s)$ for each $s$.  In particular $k(s+\ell) = r^{-2}k(s)$.  Since $k$ is increasing for an admissible curve, we must have $k(s)>0$ for all $s$, so the curve $PTX(s)$ is locally convex.  Also $T(s+\ell) = \E^{2i\theta}T(s)$ for each $s$, where $T(s)$ is the Euclidean tangent vector to the curve $PTX$ at $s$.  It follows that $\int_0^\ell k\,du = 2\theta+2\pi n$ for some integer $n$, and since $k>0$ this implies $2\pi n+2\theta>0$.  Clearly this integer $n$ is an invariant of the homotopy class.  The curve $z(s) = (r\E^{i(\theta+n\pi)})^{2s/\ell}$ provides an example of an curve with invariant $n$, which has positive (increasing) curvature provided $\theta+n\pi>0$.
    
    It remains to show that any two curves with the same value of $n$ are homotopic to the loxodromic curve given above.  We can achieve this by first parametrising by the normal angle $\phi$, which we can assume (by first applying a rotation, which does not affect the $L$-cocompactness) that $\phi$ ranges from $0$ to $2n\pi+2\theta$. The curve is then determined by its support function $u(\phi)=\left\langle z(\phi),\E^{i\phi}\right\rangle$, and a homotopy can be produced by taking the linear interpolation between the support functions of the two curves:  The radius of curvature is linear in the support function, so the property that the radius of curvature is decreasing (i.e. the curvature is increasing) holds throughout this interpolation.
\end{proof}

We illustrate canonical curves for $n=0$, $1$ and $2$ in Figure \ref{figLoxN}.

\begin{corollary}\label{cor:lbd}
    If $X_0$ is an admissible $L$-co-compact curve, there exists $\bar\ell>0$ such that $\ell(X)\leq \bar\ell$ for any $L$-co-compact curve $X$ which is homotopic to $X_0$ through $L$-co-compact curves.
    Furthermore equality is attained on a loxodromic curve.
\end{corollary}

\begin{proof}
As in the proof of the previous Corollary, we can parametrise a period of the curve $PTX$ by the normal angle $\phi\in[0,2n\pi+2\theta)$.  Then we compute, noting that $du = k^{-1}d\phi$ and $k_u = kk_\phi$
\begin{align*}
    \ell &= \int \sqrt{k_u}\,du\\
    &=\int_0^{2n\pi+2\theta} \frac{\sqrt{kk_\phi}}{k}\,d\phi\\
    &=\int_0^{2n\pi+2\theta} \sqrt{\frac{k_\phi}{k}}\,d\phi\\
    &\leq \sqrt{2n\pi+2\theta}\left(\int \partial_\phi\log k\,d\phi\right)^{1/2}\\
    &=\sqrt{(2n\pi+2\theta)\log\frac1r}\\
    &=:\bar\ell.
\end{align*}
Equality holds precisely when $k$ is an exponential function of $\phi$, which is the loxodromic case.
\end{proof}


\section{The evolution equations}\label{sec:evoleqn}

\subsection{The Inversive curve-lengthening flow}

Let $X_0:\RR\rightarrow\CC\PP^1$ be an admissible cocompact curve, for convenience parametrised by the invariant parameter so that $X_0(x+\ell)=LX_0(x)$ for all $x\in\RR$, for some $L\in PSL(2,\CC)$.  We consider the inversive length $\ell$ of a period of the curve.
For any smooth variation $X:\ [0,\varepsilon)\times \RR/\ell\ZZ$ which respects the periodicity, in the sense that $X(x+\ell,t)=LX(x,t)$ for all $x\in\RR$ and $t\in[0,\varepsilon)$, we define $V(x):=\frac{\partial X(x,t)}{\partial t}\Big|_{t=0}\in\Gamma(X_0^*T\CC\PP^1)$.  Then we can write $V(x) = f(x)\mathcal{N}(x)+\alpha(x)\mathcal{T}(x)$.  The variation equations of the previous section imply
$$
\partial_tds = \left(\frac12f^{(3)}+2Qf_s+Q_sf+\alpha_s\right)\,ds
$$
so that
$$
\frac{d}{dt}\ell = \int_0^\ell \left(\frac12f^{(3)}+2Qf_s+Q_sf+\alpha_s\right)\,ds
=-\int_0^\ell Q_s f\,ds
$$
since $\alpha(\ell)=\alpha(0)$ and $Q(\ell)=Q(0)$ by the periodicity.  Thus the direction of steepest ascent of $\ell$ in the $L^2(ds)$ sense corresponds to the case when $f = -Q_s$ and $\alpha=0$:
\begin{equation}\label{eq:ICLF}
\partial_tX = -Q_s\mathcal{N}.
\end{equation}

The flow \eqref{eq:ICLF} is called the \emph{inversive curvature flow} and is our
main object of study.
It is a sixth-order degenerate nonlinear singular system of parabolic
differential equations. From the variation equations of the previous section we deduce the following evolution equation for the fundamental invariant $Q$:
\begin{align}
    \partial_tQ &= \frac14Q^{(6)}+2QQ^{(4)}+\frac52Q_sQ^{(3)}+(4Q^2+2Q_{ss}+1)Q_{ss}+(2QQ_s+\frac12Q^{(3)})Q_s\notag\\
    &=\frac14Q^{(6)}+2QQ^{(4)}+3Q_sQ^{(3)}+2Q_{ss}^2+(4Q^2+1)Q_{ss}+2QQ_s^2.
    \label{eq:evolQ}
\end{align}


Evolution equations for all higher invariants may be computed from Equation \eqref{eq:evolQ} by successive differentiation, making using of the commutator expression
$$
[\partial_t,\partial_s]= \mathcal{C}\partial_s,\quad \mathcal{C} = \frac12Q^{(4)}+2QQ_{ss}+Q_s^2.
$$

\subsection{\texorpdfstring{$L^2$}{L2} bound on the invariant curvature}

\begin{proposition}
Let $x_0:\RR\rightarrow\CC\PP^1$ be a cocompact admissible curve.
Along the inversive curve lengthening flow $X:\RR\times[0,T)\rightarrow\CC\PP^1$ with initial data $X_0$ we have the estimate
\[
\vn{Q}_2^2 + \frac1{22}  \int_0^t\vn{Q_{s^3}}_2^2\,d\tau
           + \frac{5}{4}\int_0^t\vn{Q_sQ}_2^2   \,d\tau
           + 2           \int_0^t\vn{Q_s}_2^2    \,d\tau
\le \vn{Q}_2^2\Big|_{t=0}\,.
\]
\label{LMQbd}  
\end{proposition}
\begin{proof}
From the formula \eqref{eq:evolQ} and $\partial_tds = -\mathcal{C}ds$, we compute
\begin{align*}
\frac{d}{dt}\int Q^2 \,ds
 &=  \int 2Q\partial_tQ-\mathcal{C}Q^2\,ds\\
 &= 2\int Q\left(\frac14Q^{(6)}+2QQ^{(4)}+3Q_sQ^{(3)}+2Q_{ss}^2+(4Q^2+1)Q_{ss}+2QQ_s^2\right)\,ds\\
 &\quad\null -\int Q^2\left(\frac12Q^{(4)}+2QQ_{ss}+Q_s^2\right)\,ds,
 \end{align*}
 where the integrals are understood to be taken over one period of the curve.
 Using the identities
 \begin{align*}
 \int QQ^{(6)}\,ds&=-\int (Q^{(3)})^2\,ds;\\
 \int QQ_{ss}\,ds&=-\int Q_s^2\,ds;\\
 \int Q^2Q^{(4)}\,ds &= -2\int QQ_sQ^{(3)}\,ds;\\
 \int Q^3Q_{ss}\,ds &= -3\int Q^2Q_s^2\,ds;\quad\text{and}\\
 \int QQ_{ss}^2\,ds &= - \int QQ_sQ^{(3)}\,ds-\int Q_s^2Q_{ss}\,ds = - \int QQ_sQ^{(3)}\,ds,
  \end{align*}
we arrive at the expression
 \begin{equation*}
 \frac{d}{dt}\int Q^2 \,ds
 =
    -  \frac12\int\left(Q^{(3)}\right)^2    \,ds
    -  15     \int Q^2Q_s^2     \,ds
    -  2      \int Q_{ss}^2        \,ds
    -  5      \int Q^{(3)}Q_sQ  \,ds
\,.
\end{equation*}
Young's inequality gives us the estimate
\[
    -5                        \int Q^{(3)}Q_sQ  \,ds
\le \frac{5}{11}             \int\left(Q^{(3)}\right)^2    \,ds
  + \frac{5}{\frac{4}{11}}   \int Q_s^2Q^2     \,ds
=   \frac{5}{11}             \int\left(Q^{(3)}\right)^2    \,ds
  + \Big(15-\frac{5}{4}\Big)\int Q_s^2Q^2     \,ds
\,.
\]
Combining this with the previous identity, we conclude
\[
\frac{d}{dt}\vn{Q}_2^2 \le -\frac1{22}\vn{Q_{s^3}}_2^2 - \frac{5}{4}\vn{Q_sQ}_2^2 - 2\vn{Q_s}_2^2
\,.
\]
Integrating this equation gives the stated estimate.
\end{proof}

\subsection{Bounds on higher derivatives}
\begin{proposition}\label{prop:derivbounds}
Let $X_0:\RR\rightarrow\CC\PP^1$ be an $L$-cocompact admissible curve.  Then for each $p\geq 1$ there exists $c$ depending on $L$, $p$, the homotopy class of $X_0$ in the space of $L$-cocompact curves, $\ell(X_0)$ and $\int Q^2\,ds\Big|_{t=0}$ such that along the inversive curve lengthening flow $X:\RR\times[0,T)\rightarrow\CC\PP^1$ with initial data $X_0$ we have
\[
	\int \left(Q^{(p)}\right)^2\,ds \leq c\min\left\{\int \left(Q^{(p)}\right)^2\,ds\Big|_{t=0}, 1+t^{-\frac{p}{3}}
    \right\}.
\]
\label{LMQsbd}
\end{proposition}
\begin{proof}
We begin by deriving the form of the evolution equation for derivatives of $Q$:

\begin{lemma}\label{lem:evolQp}
Under the inversive curve lengthening flow \eqref{eq:ICLF}, for any $p\geq 0$ we have
\begin{equation}\label{eq:evol-Qp}
\partial_tQ^{(p)} = \frac14 Q^{(p+6)}+Q^{(p+2)}+\sum_{i+j=p+4}c_{p,i,j}Q^{(i)}Q^{(j)}
+\sum_{i+j+k=p+2}c_{p,i,j,k}Q^{(i)}Q^{(j)}Q^{(k)}
\end{equation}
for some constants $c_{p,i,j}$ and $c_{p,i,j,k}$.
\end{lemma}

\begin{proof}
    We proceed by induction on $p$, noting that the Equation \eqref{eq:evolQ} has the form of Equation \eqref{eq:evol-Qp} for $p=0$.  The induction step proceeds as follows:  Assuming the form given in the Lemma holds for some $p$, we have
    \begin{align*}
\partial_tQ^{(p+1)}&=\partial_s\partial_tQ^{(p)}+\mathcal{C}\partial_sQ^{(p)}\\
    &=\partial_s\left(\frac14 Q^{(p+6)}+Q^{(p+2)}+\sum_{i+j=p+4}c_{p,i,j}Q^{(i)}Q^{(j)}
+\sum_{i+j+k=p+2}c_{p,i,j,k}Q^{(i)}Q^{(j)}Q^{(k)}
    \right)\\
    &\quad\null+\left(\frac14Q^{(4)}+2QQ_{ss}+Q_s^2\right)Q^{(p+1)}\\
&=\frac14Q^{(p+1+6)}+Q^{(p+1+2)}+\sum_{i+j=p+4}c_{p,i,j}\left(Q^{(i+1)}Q^{(j)}+Q^{(i)}Q^{(j+1)}\right)\\
&\quad\null+\sum_{i+j+k=p+2}c_{p,i,j,k}\left(Q^{(i+1)}Q^{(j)}Q^{(k)}+Q^{(i)}Q^{(j+1)}Q^{(k)}+Q^{(i)}Q^{(j)}Q^{(k+1)}
\right)\\
&\quad\null+\frac14Q^{(4)}Q^{(p+1)}+2QQ_{ss}Q^{(p+1)}+Q_s^2Q^{(p+1)}.
    \end{align*}
This completes the proof since the quadratic and cubic terms all have the form required for the $p+1$ case of Equation \eqref{eq:evol-Qp}.
\end{proof}

The result of Lemma \ref{lem:evolQp} allows us to compute evolution equations integrals of higher derivatives:

\begin{lemma}\label{lem:evol-intQp}
Under the inversive curve lengthening flow \eqref{eq:ICLF} of an admissible co-compact curve,
\begin{align}
\frac{d}{dt}\int \left(Q^{(p)}\right)^2\,ds &\leq -\frac12\int \left(Q^{(p+3)}\right)^2\,ds -2\int \left(Q^{(p+1)}\right)^2\,ds\notag\\
&\quad\null+C\!\!\!\!\!\sum_{\substack{i+j+k = 2p+4\\0\leq i,j,k\leq p+2}}\int\left|Q^{(i)}Q^{(j)}Q^{(k)}\right|\,ds
+C\!\!\!\!\!\sum_{i+j+k = p+2}
\int\left|Q^{(i)}Q^{(j)}Q^{(k)}Q^{(p)}\right|\,ds,\label{eq:ddtQp}
\end{align}
for some $C$ depending only on $p$.
\end{lemma}
\begin{proof}
From Lemma \ref{lem:evolQp} and $\partial_tds = -\mathcal{C}\,ds$ we derive
\begin{align*}
    \frac{d}{dt}\int\left(Q^{(p)}\right)^2\,ds &=\int 2Q^{(p)}\partial_tQ^{(p)}-\mathcal{C}\left(Q^{(p)}\right)^2\,ds\\
    &=-\frac12\int\left(Q^{(p+3)}\right)^2\,ds-2\int\left(Q^{(p+1)}\right)^2\,ds\\
    &\quad\null+\sum_{i+j=p+4}\tilde c_{p,i,j}\int Q^{(p)}Q^{(i)}Q^{(j)}\,ds
    +\sum_{i+j+k=p+2}\tilde c_{p,i,j,k}\int Q^{(p)}Q^{(i)}Q^{(j)}Q^{(k)}\,ds.
\end{align*}
The last term only involves derivatives of $Q$ up to order $p+2$, so can be bounded by the last term of \eqref{eq:ddtQp}.  The penultimate term is also of the correct form except that it may contain terms of the form $\int Q^{(p)}QQ^{(p+4)}$ or $\int Q^{(p)}Q_sQ^{(p+3)}\,ds$.  These can be integrated by parts to produce terms of the form $\int Q(Q^{(p+2)})^2\,ds$, $\int Q_sQ^{(p+1)}Q^{(p+2)}\,ds$ and $\int Q^{(p)}Q_ssQ^{(p+2)}$, all of which can be estimated by the penultimate term in \eqref{eq:ddtQp}.
\end{proof}

Now we complete the proof of Proposition \ref{prop:derivbounds}.  The strategy is to apply Gagliardo-Nirenberg inequalities to control the difficult nonlinear terms in terms of the good leading term, in a manner inspired by the work of Polden \cite{Polden}.  This method was also successfully applied to the analysis of the affine curve lengthening flow in \cite{ACLF}.  A convenient exposition of these inequalities and their application is given in \cite{DKS01}.
In the present situation, since the length $\ell$ is bounded above and below along the evolution, we have uniform control on the constants in the Gagliardo-Nirenberg interpolation inequality.  We estimate the terms on the second line of \eqref{eq:ddtQp}:
\begin{align*}
\int |Q^{(i)}Q^{(j)}Q^{(k)}|ds&\leq \|Q^{(i)}\|_{L^3}\|Q^{(j)}\|_{L^3}\|Q^{(k)}\|_{L^3}\\
&\leq C\|Q\|_{L^2}^{\frac{3p+7-i-j-k}{p+5/2}}\|Q^{(p+3)}\|^{\frac{i+j+k+1/2}{p+5/2}}\\
&\leq C\|Q^{(p+3)}\|_{L^2}^{\frac{4p+9}{2p+5}}\\
&\leq \varepsilon\int\left(Q^{(p+3)}\right)^2\,ds + C(\varepsilon)
\end{align*}
for any $\varepsilon>0$, where we used Young's inequality in the last step, noting that the exponent $\frac{4p+9}{2p+5}$ is strictly less than $2$.  Similarly we have
\begin{align*}
\int |Q^{(i)}Q^{(j)}G^{(k)}Q^{(p)}|\,ds &\leq C\left(\int Q^2\,ds\right)^{\frac{2p+9}{2p+6}}\left(\int\left(Q^{(p+3)}\right)^2\,ds\right)^{\frac{2p+3}{2p+6}}\\
&\leq \varepsilon\int\left(Q^{(p+3)}\right)^2\,ds+C(\varepsilon),
\end{align*}
again using the observation that $\frac{2p+3}{2p+6}<1$.  This gives
$$
\frac{d}{dt}\int\left(Q^{(p)}\right)^2\,ds\leq -\frac14\int\left(Q^{(p+3)}\right)^2\,ds+C.
$$
The Gagliardo-Nirenberg inequality again provides that
$$
\int\left(Q^{(p)}\right)^2\,ds \leq C\left(\int Q^2\,ds\right)^\frac{3}{p+3}\left(\int\left(Q^{(p+3)}\right)^2\,ds\right)^\frac{p}{p+3},
$$
implying that
$$
\frac{d}{dt}\int\left(Q^{(p)}\right)^2\,ds \leq -C_1\left(\int\left(Q^{(p)}\right)^2\,ds\right)^{\frac{p+3}{p}}+C_2,
$$
from which we conclude the bound in the Proposition.\end{proof}

\subsection{Long time existence}

\begin{proposition}\label{prop:LTE}
Let $X_0:\RR\rightarrow\CC\PP^1$ be an $L$-cocompact admissible curve.
The inverse curve lengthening flow $X:\RR\times[0,T)\rightarrow\CC\PP^1$ with
initial data $X_0$ exists for all time.
\label{PRglobal}
\end{proposition}
\begin{proof}
Suppose the maximal time of existence $T$ is finite.
By Proposition \ref{prop:derivbounds} we have $\vn{Q^{(p)}}_2^2$ bounded uniformly for all $p\geq 0$,
which implies that $\vn{Q^{(p)}}_\infty$ is bounded uniformly for all $p\geq 0$.

\begin{lemma}
For each $p,q\geq 0$, $\|\partial_t^q\partial_u^pG\|_\infty$ is bounded on $[0,T)$, where $u$ is the parameter on $\RR$.
\end{lemma}

\begin{proof}
    Fix $u_0\in\RR$, and define a global arc length parameter by $s(u,t) = \int_{u_0}^uds$.  Then we have $s(u_0,t)=0$ for all $t$, and $\partial_t\frac{\partial s}{\partial u} = -\mathcal{C}\frac{\partial s}{\partial u}$, and it follows that $\frac1C\leq \frac{\partial s}{\partial u}\leq C$ on $\RR\times[0,T)$, where $C = \exp(T\|\mathcal{C}\|_\infty)$.  Differentiating successively with respect to $u$, we further derive by induction on $p$ the identity
    $$
    \partial_t\frac{\partial^ps}{\partial u^p} = \sum_{q=0}^{p-1}\mathcal{C}^{(q)}\sum_{\substack{i+1+\ldots+i_{q+1}=p\\i_1,\ldots,i_{q+1}\geq 1}}c_{q,i_1,\ldots,i_{q+1}}\frac{\partial^{i_1}s}{\partial u^{i_1}}\ldots\frac{\partial^{i_{q+1}}s}{\partial u^{i_{q+1}}}
    $$
    for some constants $c_{q,i_1,\ldots,i_{q+1}}$.  From this and the bounds on $Q^{(p)}$ for all $p$ (hence on $\mathcal{C}^{(p)}$ for all $p$) from Proposition \ref{prop:derivbounds} we deduce bounds on $\frac{\partial^ps}{\partial u^p}$ for each $p$ on $[0,\ell]\times[0,T]$.  It follows that $s$ is a smooth function of $u$, and vice versa.

    The bounds on $Q^{(p)}$ and the identities \eqref{eq:SF} and \eqref{eq:dGdt}, together with the bounds on $\frac{\partial^ps}{\partial u^p}$ derived above, imply the required bounds on $\partial_t^p\partial_u^qG$ for any $p,q\geq 0$. 
\end{proof}

It follows from the Lemma that $G(\cdot,t)$ converges in $C^\infty$ to a smooth limit $G(.,T)$ as $t\to T$, and in particular $X(.,t) = G^{-1}\left(\begin{bmatrix}0\\1\end{bmatrix}\right)$ also converges smoothly to a limit $X(\cdot,T)$ which is also $L$-co-compact and admissible.  Applying the short time existence result with initial condition $X(.,T)$ then extends the solution smoothly to a longer time interval, contradicting the maximality of $T$.  Therefore the interval of existence is infinite.
\end{proof}

\subsection{Exponential convergence}

In this section we prove that the solution of \eqref{eq:ICLF} converges smoothly to a loxodromic curve as $t\to\infty$, at an exponential rate.  We first note that $\|Q_s\|_{L^2}$ must approach zero along a sequence of times:

\begin{lemma}
Along the ICLF we have
\[
\int_0^\infty\int Q_s^2\,ds\,dt < \bar\ell
\]
\label{LMQsinL1}
where $\bar\ell$ is the constant derived in Corollary \ref{cor:lbd}.
\end{lemma}

\begin{proof}
Since $\frac{d\ell}{dt} = \int Q_s^2\,ds$, we have for any $t\in(0,\infty)$ that 
$\bar\ell \ge \ell(t) \ge \ell(t) - \ell(0) =
\int_0^t\int Q_s^2\,ds\,d\hat t$.
The result follows on taking $t\to\infty$.
\end{proof}

It follows that there exists a sequence $\{t_j\}$ approaching infinity
such that $\vn{Q_s}_{L^2}(t_j)$ converges to zero.
We exploit this to deduce exponential convergence:

\begin{lemma}\label{lem:Qs-exp-decay}
Let $X_0:\RR\rightarrow\CC\PP^1$ be an admissible $L$-cocompact curve. Then
the solution $X:\ \RR\times[0,\infty)\to\CC\PP^1$ of the inversive curve lengthening flow \eqref{eq:ICLF} satisfies
	\[
		\int Q_s^2\,ds \le Ce^{-\lambda t}
	\]
	for some positive constants $C$ and $\lambda$ depending only on $\ell(X_0)$, $\bar\ell$ and $\int Q^2\,ds\Big|_{t=0}$.
\end{lemma}
\begin{proof}
By Proposition \ref{PRglobal} we have $T = \infty$.
Lemma \ref{LMQsinL1} implies that for any $\varepsilon>0$ there exists a $t_0\in[0,\bar\ell\varepsilon^{-1}]$ such that
	$\vn{Q_s}_2^2(t_0) \leq \varepsilon$.
Computing more carefully in the case $p=1$ in the proof of Lemma \ref{lem:evol-intQp} we find
\begin{align*}
\frac{d}{dt}\int Q_{s}^2\,ds
 &=  -\frac12\int\left(Q^{(4)}\right)^2\,ds -2\int Q_{ss}^2\,ds-2\int QQ_{ss}Q^{(4)}\,ds-8\int Q^2Q_{ss}^2\,ds\\
 &\quad\null -\int Q_{ss}^3\,ds+\frac73\int Q_s^4\,ds+\frac12 \int Q_s^2Q^{(4)}\,ds\\
 &\leq-\frac38\int\left(Q^{(4)}\right)^2\,ds-2\int Q_{ss}^2
 -\int Q_{ss}^3\,ds+\frac73\int Q_s^4\,ds+\frac12 \int Q_s^2Q^{(4)}\,ds\\
 &\leq-\frac3{16}\int\left(Q^{(4)}\right)^2\,ds-2\int Q_{ss}^2\,ds
 -\int Q_{ss}^3\,ds+\frac83\int Q_s^4\,ds,
\end{align*}
where we completed the square to eliminate the second and third terms in the first inequality, and then expanded the last term using Young's inequality in the last line.  It remains to estimate the final two terms in a useful way, exploiting the smallness of $\int Q_s^2\,ds$: Applying the Gagliardo-Nirenberg interpolation in equality we obtain
$$
\int Q_s^4\,ds \leq C\left(\int Q_s^2\,ds\right)^{\frac{11}6}\left(\int\left(Q^{(4)}\right)^2\,ds\right)^{\frac16}\leq \alpha \int \left(Q^{(4)}\right)^2\,ds + C(\alpha)\left(\int Q_s^2\,ds\right)^{\frac{11}{5}}
$$
for any $\alpha>0$.  Similarly we find
$$
\int|Q_{ss}|^3\,ds \leq C\left(\int Q_s^2\,ds\right)^{\frac{11}{12}}\left(\int\left(Q^{(4)}\right)^2\,ds\right)^{\frac{7}{12}}
\leq \alpha\int\left(Q^{(4)}\right)^2\,ds+C(\alpha)\left(\int Q_s^2\,ds\right)^{\frac{11}{5}}
$$
Choosing $\alpha$ appropriately, and writing $I:=\int Q_s^2\,ds$, this gives the inequality
$$
\frac{d}{dt}I \leq -c_1I+C_2I^{\frac{11}{5}}\leq -c_1I\left(1-c_3I^{\frac65}\right).
$$
for some positive constants $c_1,c_2,c_3$.
In particular, if $I(t_0)\leq \left(\frac{1}{2c_3}\right)^{5/6}$ then we have $I(t)\leq I(t_0)$ for $t\geq t_0$, and $\frac{d}{dt}I\leq -\frac{c_1}{2}I$ for $t\geq t_0$, giving the required exponential convergence.

\end{proof}

The exponential decay allows us to upgrade the bounds for $\vn{Q^{(p)}}_2^2$ to exponential decay.

\begin{corollary}
Let $X_0:\RR\rightarrow\RR^2$ be an admissible  $L$-cocompact curve.
Along the inversive curve lengthening flow $X:\RR\times[0,T)\rightarrow\CC\PP^1$ with initial data $X_0$ we have for all $p\in\NN$
	\[
		\int \left(Q^{(p)}\right)^2\,ds \le be^{-\lambda t/2}
	\]
where $b$ depends only on $p$, $\ell(0)$, $\bar\ell$ and $\int Q^2\,ds\Big|_{t=0}$.
\end{corollary}
\begin{proof}
We calculate for $p>1$
\[
	\int\left(Q^{(p)}\right)^2\,ds
	\le \vn{Q_s}_2\vn{Q^{({2p-1)} }}_2
	\le Ce^{-\lambda t/2}
\]
using the large time bound from Proposition \ref{prop:derivbounds} to bound the second factor.
\end{proof}

A modification of the proof of Proposition \ref{prop:LTE} (using the exponential decay to give estimates for all time and convergence of $Q$ to a constant) then gives the following:

\begin{corollary}\label{cor:longtime}
Let $X_0:\RR\rightarrow\CC\PP^1$ be an admissible $L$-cocompact curve.
Then the inversive curve lengthening flow $X:\RR\times[0,\infty)\rightarrow\CC\PP^1$
converges
exponentially fast in $C^p$ for every $p\geq 0$ to a limiting curve $X_\infty$ which is an $L$-cocompact loxodromic curve.
\end{corollary}

\section{Flow of irregular curves}
\label{sec:rough}

In this section we observe that our estimates allow us to consider solutions originating from non-smooth initial curves:

\begin{theorem}\label{thm:nonsmooth}
    Let $X_0:\ \RR\to\CC\PP^1$ be an admissible $L$-cocompact curve for which $\int Q^2\,ds<\infty$.  Then there exists a unique solution $X:\ \RR\times[0,\infty)\to\CC\PP^1$ of the inversive curve lengthening flow \eqref{eq:ICLF} which is smooth for $t>0$ and in $C^{0,1/12}\left([0,\infty),C^4(\RR,\CC\PP^1)\right)$ (in particular, the Gauss maps $G(u,t)$ are H\"older continuous on $\RR\times[0,\infty)$).
\end{theorem}

We remark that since the solutions immediately become smooth, the long time behaviour is as described in Corollary \ref{cor:longtime}.

\begin{proof}
We construct a solution as a limit of solutions arising from a sequence of smooth admissible $L$-cocompact initial curves $X_{i,0}$ which converge to $X_0$ as $i\to\infty$ (such approximations can be constructed by applying the heat equation to the support function of the curve, in a stereographic projection to $\RR^2$ in which $L$ acts by scaling and rotation -- note that this approximation is continuous in $H^5$ and so $\int_{X_{i,0}}Q^2\,ds$ converges to $\int_{X_0}Q^2\,ds$ as $i\to\infty$, and we note that the inversive arc length element also converges as $i\to\infty$).  

The estimates of Proposition \ref{prop:derivbounds} imply that $\|Q^{(p)}\|_{L^2(ds)}\leq C_pt^{-p/6}$ for each $p>0$, with constant $C_p$ independent of $i$.  By the Gagliardo-Nirenberg interpolation inequality, this implies that
$$
\|Q^{(p)}\|_{L^\infty}\leq C\|Q^{(p)}\|_{L^2}^{1/2}\|Q^{(p+1)}|_{L^2}^{1/2}\leq Ct^{-p/12-(p+1)/12} = Ct^{-p/6-1/12}.
$$
In particular each of the solutions $X_{i,t}$ is bounded in $C^p$ for each $p$ and each $t>0$, with bound independent of $i$.
The expression \eqref{eq:dGdt} for the time derivative of the Gauss map implies that
$$
\left|\frac{dG}{dt}\right|\leq C|G|\|Q\|_{C^5}\leq C|G|t^{-5/6-1/12} = C|G|t^{-11/12},
$$
implying that $|G(u,t)-G(u,0)|\leq Ct^{1/12}$, with $C$ independent of $i$.  It follows that (for a subsequence of $i\to\infty$) the solutions $X_{i,t}$ converge in $C^\infty_{\mathrm{loc}}((0,\infty)\times\RR,\CC\PP^1)$ to a smooth solution $X(.,t)$, and that $G(.,t)$ converges to $G(X_0)$ in $C^{0,1/12}$.
\end{proof}

\begin{remark}
  It seems reasonable to expect that solutions of the inversive curve lengthening flow might exist for any initial curve which is co-compact and admissible, a condition which could be interpreted in the sense that the osculating circles are nested along the curve.  This is equivalent to the curve being $C^{1,1}$ with increasing curvature function, while the condition in Theorem \ref{thm:nonsmooth} amounts to a much more restrictive $H^5$ condition on the curve.  At present we are not aware of a reasonable strategy to prove such a conjecture.
  \end{remark}

\begin{remark}
    A further problem of interest is to define solutions (perhaps in a suitably generalised sense) for curves which are not admissible (in particular, for suitable classes of closed curves).  We are encouraged by the result of Angenent, Sapiro and Tannenbaum \cite{AST98} which considers the affine normal flow for curves which need not be convex (that is to say, the curves need not be admissible in the context of special affine geometry).
\end{remark}

\end{document}